\documentclass[11pt]{amsart}
\usepackage[all]{xy}
\usepackage{amsmath,amsfonts,amssymb,amsthm,epsfig,amscd,psfrag,framed,float}

\newcommand{\puteps}[2][0.5]
{\includegraphics[scale=#1]{#2.eps}}

\allowdisplaybreaks
\usepackage{xargs}
\usepackage{environ}
\usepackage{graphicx}

\newtheorem{Theorem}{Theorem}[section]
 
\newtheorem{Proposition}[Theorem]{Proposition} 
\newtheorem{Lemma}[Theorem]{Lemma}

\newtheorem{Corollary}[Theorem]{Corollary}

\theoremstyle{definition}
\newtheorem{Remark}[Theorem]{Remark}

\newcommand{\smccbub}[1]{
\xybox{%
 (-5,0)*{};
  (5,0)*{};
  (-2,0)*{}="t1";
  (2,0)*{}="t2";
  "t2";"t1" **\crv{(2,3) & (-2,3)}; ?(.05)*\dir{>} ?(1)*\dir{>}; ?(.3)*\dir{}+(0,0)*{\bullet}+(1,2)*{\scs
  {#1}};
  "t2";"t1" **\crv{(2,-3) & (-2,-3)};
}}

\newcommand{\Ucapr}{\;\;
    \vcenter{\xy (-2,-1)*{}; (2,-1)*{} **\crv{(-2,3) & (2,3)}?(1)*\dir{>};
            (2,-3)*{};(-2,3)*{}; \endxy} \;\; }

\newcommand{\sUup}{
    \xy {\ar (0,-2)*{};(0,2)*{} };(1.5,0)*{};(-1.5,0)*{};\endxy}

\newcommand{\sUdown}{
    \xy {\ar (0,2)*{};(0,-2)*{} };(1.5,0)*{};(-1.5,0)*{};\endxy}

\newcommand{\sUupdot}{
   \xy {\ar (0,-2)*{};(0,2)*{} };(0,-0.5)*{\scs \bullet};(1.5,0)*{};(-1.5,0)*{};\endxy}

\newcommand{\sUdowndot}{
   \xy {\ar (0,2)*{};(0,-2)*{} };(0,0.5)*{\scs \bullet};(1.5,0)*{};(-1.5,0)*{};\endxy}

\newcommand{\scs}{\scriptstyle}

\newcommand{\qbins}[2]{
\left[
 \begin{array}{c}
 \scs #1 \\
 \scs #2 \\
 \end{array}
 \right]
}

\def\p{{\partial}}
\def\K{{\mathcal{K}}}
\def\hbeta{\widehat{\beta}}

\def\Spec{{\mbox{Spec }}}

\def\Sym{{\mbox{Sym}}}

\def\adj{{\rm{adj}}}

\def\0{{\underline{0}}}
\def\u2{{\underline{2}}}

\def\uk{{\underline{k}}}

\def\A{{\mathbb{A}}}

\def\Kom{{\sf{Kom}}}
\def\1{{\mathbf{1}}}
\def\P{{\sf{P}}}

\def\k{{\Bbbk}}
\def\id{{\mathrm{id}}}
\def\l{{\lambda}}

\def\sl{\mathfrak{sl}}
\def\gl{\mathfrak{gl}}
\def\g{\mathfrak{g}}

\newcommand{\Z}{\mathbb{Z}}
\newcommand{\C}{\mathbb{C}}

\def\O{{\mathcal O}}

\def\sfA{{\sf{A}}}
\def\sfB{{\sf{B}}}
\def\sA{{\mathcal{A}}}
\def\sB{{\mathcal{B}}}
\def\sE{{\mathcal{E}}}

\def\sF{{\mathcal{F}}}
\def\sL{{\mathcal{L}}}

\def\sH{{\mathcal{H}}}
\def\sT{{\mathcal{T}}}
\def\sP{{\mathcal{P}}}
\def\sM{{\mathcal{M}}}
\def\sN{{\mathcal{N}}}
\def\sQ{{\mathcal{Q}}}

\def\t{{\theta}}

\def\N{\mathbb N}

\newcommand{\E}{\mathsf{E}}
\newcommand{\F}{\mathsf{F}}

\newcommand{\T}{\mathsf{T}}
\newcommand{\la}{\langle}
\newcommand{\ra}{\rangle}
\newcommand{\Ext}{\operatorname{Ext}}
\newcommand{\Tor}{\operatorname{Tor}}
\newcommand{\Hom}{\operatorname{Hom}}
\newcommand{\End}{\operatorname{End}}

\newcommand{\spn}{\operatorname{span}}

\begin{document}
\setcounter{tocdepth}{1}

\title[Remarks on coloured triply graded link invariants]{Remarks on coloured triply graded link invariants}

\author{Sabin Cautis}
\email{cautis@math.ubc.ca}
\address{Department of Mathematics\\ University of British Columbia \\ Vancouver BC, Canada}

\begin{abstract}
We explain how existing results (such as categorical $\sl_n$ actions, associated braid group actions and infinite twists) can be used to define a triply graded link invariant which categorifies the HOMFLY polynomial of links coloured by arbitrary partitions. The construction uses a categorified HOMFLY clasp defined via cabling and infinite twists. We briefly discuss differentials and speculate on related structures. 
\end{abstract}

\maketitle
\tableofcontents
\section{Introduction}

In \cite{KR1,K} Khovanov and Rozansky defined a triply graded link invariant using matrix factorizations and subsequently Soergel bimodules. In their case the link is coloured by the partition $(1)$ and the invariant categorifies the HOMFLY polynomial. In this paper we explain how existing tools can be used to extend this construction to links coloured by arbitrary partitions which categorifies the coloured HOMFLY polynomial. 

The idea is as follows. First one defines a 2-category $\K_n$ out of Soergel bimodules and constructs a categorical $(\sl_n,\t)$ action on it (Sections \ref{sec:gactions} and \ref{sec:Kn}). Combining this action with a trace 2-functor (Hochschild (co)homology) one obtains a triply graded invariant for links coloured by partitions with only one part $(k)$ for $k \in \N$. 

Finally, to deal with an arbitrary partition $(k_1, \dots, k_i)$ one cables together $i$ strands labeled $k_1, \dots, k_i$ and composes with a certain projector $\P^-$. We will call these (categorified) HOMFLY clasps to differentiate them from those in the Reshetikhin-Turaev context (RT clasps). Apart from some general results on $(\sl_n,\t)$ actions and associated braid group actions and projectors (see for instance \cite{C1}) I have tried to keep this paper self-contained. Some example computations are worked out in Section \ref{sec:examples}.  

There are many papers in the literature on coloured HOMFLY homology and it is difficult to list them all without forgetting some. We try to recall some of the ones which are more closely related to this paper. 

There are several papers defining various generalities of triply graded homologies. In \cite{MSV} Mackaay, Sto\v{s}i\'{c} and Vaz work out the case of links labeled by the one part partition $(2)$. In \cite{WW} Webster and Williamson define a triply graded homology of links coloured by partitions with only one part. Their construction, which is geometric, is related to ours via the equivalence between perverse sheaves on finite flag varieties and (singular) Soergel bimodules. The same relationship appears (and is briefly discussed) in \cite{CDK}. More recently, Wedrich \cite{W} examines these constructions in the ``reduced'' case as well as some associated spectral sequences.  

The papers \cite{H,AH} discuss the categorified HOMFLY clasps for partitions with parts of size at most one (i.e. coloured with $(1^k)$ for $k \in \N$). As with our projectors, these are build as infinite twists. As far as I understand, Elias and Hogancamp aim to develop a more systematic, more general construction of such projectors. This will hopefully shed some light on the projectors $\P^-$ and the various properties they (are expected to) satisfy. 

In \cite{DGR,Ra} it was conjectured (and partially proved) that there exist certain differentials on triply graded link homology which recovers $SL(N)$ link homology. In Section \ref{sec:differentials} we discuss a differential $d_N$ for $N > 0$ which gives rise to an $SL(N)$-type link invariant. Somewhat surprising, the resulting homology seems to be finite-dimensional while categorifying $SL(N)$-representations of the form $\Sym^k(\C^N)$. A homology of this form is predicted by the physical interpretation of knot homologies as spaces of open BPS states (see for instance \cite{GS}) but does not show up in our earlier work on knot homologies (\cite{CK1} and subsequent papers). In Section \ref{sec:speculation} we also speculate on defining differentials $d_N$ for $N < 0$ which should categorify $SL(N)$-representations of the form $\Lambda^k(\C^N)$. 

The author was supported by an NSERC discovery/accelerator grant.

\section{Background: $(\sl_n,\t)$ actions, braid group actions and projectors}\label{sec:gactions}

\subsection{Notation}\label{sec:notation}
We work over an arbitrary field $\k$. By a graded 2-category $\K$ we mean a 2-category whose 1-morphisms are equipped with an auto-equivalence $\la 1 \ra$ (so graded means $\Z$-graded).  We say $\K$ is idempotent complete if for any 2-morphism $f$ with $f^2=f$ the image of $f$ exists in $\K$. 

For $n \ge 1$ we denote by $[n]$ the quantum integer $q^{n-1} + q^{n-3} + \dots + q^{-n+3} + q^{-n+1}$ where $q$ is a formal variable. By convention, for negative entries we let $[-n] = -[n]$. Moreover $[n]! := [1][2] \dots [n]$ and 
$\qbins{n}{k} := \frac{[n]!}{[k]![n-k]!}.$

If $f = f_aq^a \in {\N}[q,q^{-1}]$ and ${\sfA}$ is a 1-morphism inside a graded 2-category $\K$ then we write $\oplus_f {\sfA}$ for the direct sum $\oplus_{s \in \Z} {\sfA}^{\oplus f_s} \la s \ra$. For example, 
$$\bigoplus_{[n]} {\sfA} = {\sfA} \la n-1 \ra \oplus {\sfA} \la n-3 \ra \oplus \dots \oplus {\sfA} \la -n+3 \ra \oplus {\sfA} \la -n+1 \ra.$$
We will always assume $\N$ contains $0$. Moreover, we will write $\End^i({\sfA})$ as shorthand for $\Hom({\sfA}, {\sfA} \la i \ra)$ where $i \in \Z$. 

Finally, if $\gamma_i: \sfA_i \rightarrow \sfB_i$ is a sequence of 2-morphisms in $\K$ for $i=1, \dots, k$ we will write $\gamma_1 \dots \gamma_k: \sfA_1 \dots \sfA_k \rightarrow \sfB_1 \dots \sfB_k$ for the corresponding 2-morphism. We will denote by $I$ the identity 2-morphism. 

\subsection{Categorical actions}
$(\g,\t)$ actions were introduced in \cite{C2} in order to simplify some of the earlier definitions from \cite{KL,Rou,CL}. A $(\g,\t)$ action involves a target graded, additive, $\k$-linear, idempotent complete 2-category $\K$ whose objects are indexed by the weight lattice of $\g$. 

In this paper we only consider the case $\g=\sl_n$. The vertex set of the Dynkin diagram of $\sl_n$ is indexed by $I = \{1, \dots, n-1\}$. However, it will be more convenient if the objects $\K(\uk)$ of $\K$ are indexed by $\uk = (k_1, \dots, k_n) \in \Z^n$ which we can identify with the weight lattice of $\gl_n$. In this notation the root lattice is generated by $\alpha_i = (0,\dots,-1,1,\dots,0)$ for $i \in I$ (this notation agrees with that in \cite{C1}). We equip $\Z^n$ with the standard non-degenerate bilinear form $\la \cdot,\cdot \ra: \Z^n \times \Z^n \rightarrow \Z$ (so that $\la \alpha_i, \alpha_j \ra$ is given by the standard Cartan datum for $\gl_n$).  

We require that the 2-category $\K$ is equipped with the following. 
\begin{itemize}
\item 1-morphisms: $\E_i \1_\uk = \1_{\uk+\alpha_i} \E_i$ and $\F_i \1_{\uk+\alpha_i} = \1_\uk \F_i$ where $\1_\uk$ is the identity 1-morphism of $\K(\uk)$.
\item 2-morphisms: for each $\uk \in \Z^n$, a $\k$-linear map $\spn \{\alpha_i: i \in I\} \rightarrow \End^2(\1_\uk)$.
\end{itemize}
We abuse notation and denote by $\theta \in \End^2(\1_\uk)$ the image of $\theta \in \spn \{\alpha_i: i \in I\}$ under the linear maps above. On this data we impose the following conditions.
\begin{enumerate}
\item \label{co:hom1} $\End^l(\1_\uk)$ is zero if $l < 0$ and one-dimensional if $l=0$ and $\1_\uk \ne 0$. Moreover, the space of maps between any two 1-morphisms is finite-dimensional.
\item \label{co:adj} $\E_i$ and $\F_i$ are left and right adjoints of each other up to specified shifts. More precisely:
\begin{enumerate}
\item $(\E_i \1_\uk)^R \cong \1_\uk \F_i \la \la \uk, \alpha_i \ra \ra + 1 \ra$
\item $(\E_i \1_\uk)^L \cong \1_\uk \F_i \la - \la \uk, \alpha_i \ra -1 \ra$.
\end{enumerate}

\item \label{co:EF} We have
$$
\begin{cases}
\E_i \F_i \1_\uk \cong \F_i \E_i \1_\uk \bigoplus_{[\la \uk, \alpha_i \ra]} \1_\uk &  \text{ if } \la \uk, \alpha_i \ra \ge 0 \\
\F_i \E_i \1_\uk \cong \E_i \F_i \1_\uk \bigoplus_{[-\la \uk, \alpha_i \ra]} \1_\uk &  \text{ if } \la \uk, \alpha_i \ra \le 0. \end{cases}
$$

\item \label{co:EiFj} If $i \ne j \in I$, then $\F_j \E_i \1_\uk \cong \E_i \F_j \1_\uk$.

\item \label{co:theta}For $i \in I$ we have
$$\E_i \E_i \cong \E_i^{(2)} \la -1 \ra \oplus \E_i^{(2)} \la 1 \ra$$
for some 1-morphism $\E_i^{(2)}$. Moreover, if $\theta \in \spn \{\alpha_i: i \in I\}$ then the map $I \t I \in \End^2(\E_i \1_\uk \E_i)$ induces a map between the summands $\E_i^{(2)} \la 1 \ra$ on either side which is
\begin{itemize}
\item nonzero if $\la \theta, \alpha_i \ra \ne 0$ and
\item zero if $\la \theta, \alpha_i \ra = 0$.
\end{itemize}

\item \label{co:vanish1} If $\alpha = \alpha_i$ or $\alpha = \alpha_i + \alpha_j$ for some $i,j \in I$ with $|i-j|=1$, then $\1_{\uk+r \alpha} = 0$ for $r \gg 0$ or $r \ll 0$.

\item \label{co:new} Suppose $i \ne j \in I$. If $\1_{\uk+\alpha_i}$ and $\1_{\uk+\alpha_j}$ are nonzero, then $\1_\uk$ and $\1_{\uk+\alpha_i+\alpha_j}$ are also nonzero.

\end{enumerate}

\begin{Remark} 
In the rest of the paper the object $\K(\uk)$ will be nonzero (i.e. $\1_\uk \ne 0$) if and only if all $k_i \ge 0$. Thus, conditions (\ref{co:vanish1}) and (\ref{co:new}) are trivial to check. 
\end{Remark}

In \cite[Theorem 1.1]{C2} we showed that such an $(\sl_n,\t)$ action must carry an action of the quiver Hecke algebras (KLR algebras). In particular, this gives us a decomposition 
$$\E_i^r \cong \bigoplus_{[r]!} \E_i^{(r)} \ \ \text{ and } \ \ \F_i^r \cong \bigoplus_{[r]!} \F_i^{(r)}$$
for certain 1-morphism $\E_i^{(r)}$ and $\F_i^{(r)}$ (called divided powers). These satisfy
\begin{align*}
(\E^{(r)}_i \1_\uk)^R &\cong \1_\l \F^{(r)}_i \la r(\la \uk, \alpha_i \ra + r) \ra \\
(\E^{(r)}_i \1_\uk)^L &\cong \1_\l \F^{(r)}_i \la -r(\la \uk, \alpha_i \ra + r) \ra.
\end{align*}

\subsection{(Categorical) braid group actions}\label{sec:braid}

The reason we are interested in $(\sl_n,\t)$ actions is that they can be used to define braid group actions \cite{CK2}, as we now recall. 

Suppose that, as above, we have an $(\sl_n,\t)$ action on a 2-category $\K$. Denote by $\Kom(\K)$ the bounded homotopy category of $\K$ (where objects are the same as in $\K$, 1-morphisms are complexes of 1-morphisms which are bounded from above and below and 2-morphisms are maps of complexes). We define $\T_i \1_\uk \in \Kom(\K)$ as 
\begin{align*}
\left[ \dots \rightarrow \E_i^{(-\la \uk, \alpha_i \ra+2)} \F_i^{(2)} \la -2 \ra \rightarrow \E_i^{(-\la \uk, \alpha_i \ra+1)} \F_i \la -1 \ra \rightarrow \E_i^{(-\la \uk, \alpha_i \ra)} \right] \1_\uk & \text{ if } \la \uk, \alpha_i \ra \le 0 \\
\left[\dots \rightarrow \F_i^{(\la \uk, \alpha_i \ra+2)} \E_i^{(2)} \la -2 \ra \rightarrow \F_i^{(\la \uk, \alpha_i \ra+1)} \E_i \la -1 \ra \rightarrow \F_i^{(\la \uk, \alpha_i \ra)} \right] \1_\uk & \text{ if } \la \uk, \alpha_i \ra \ge 0.
\end{align*}
One can show the differentials must be the unique nonzero maps. Notice that these complexes are actually bounded on the left since $\1_{\uk \pm r\alpha_i} = 0$ if $r \gg 0$. The main result of \cite{CK2} states that these complexes give us a braid group action. This fact categorifies a classical result of Lusztig \cite[5.2.1]{L}. 

\subsection{Categorified projectors}

To obtain projectors let us first consider 
$$\T_\omega \1_\uk := (\T_{n-1})(\T_{n-2}\T_{n-1}) \dots (\T_2 \dots \T_{n-1})(\T_1 \dots \T_{n-1}) \1_\uk$$
corresponding to a half twist in the braid group. In \cite[Section 5.3]{C1} we constructed a natural map $\1_\uk \rightarrow \T_\omega^2 \1_\uk$ and showed that there is a well-defined limit $\P^- \1_\uk := \lim_{\ell \rightarrow \infty} \T_\omega^{2 \ell} \1_\uk$ which lives in a certain subcategory $\Kom^-_*(\K) \subset \Kom^-(\K)$ of the bounded above homotopy category (see \cite[Section 3.5]{C1} for more details). 

\begin{Remark}
To illustrate, if $\K$ was the category of $\Z$-graded $\k$-vector spaces then $\oplus_{i \ge 0} \k [i]\la -i \ra$ would belong to $\Kom^-_*(\K)$ because $\sum_{i \ge 0} (-1)^i q^i [\k]$ converges to $\frac{1}{1+q} [\k]$ (here $[\k]$ is the class in K-theory of the one-dimensional vector space). On the other hand, $\oplus_{i \ge 0} \k [i]$ would not belong to $\Kom^-_*(\K)$ because $\sum_{i \ge 0} (-1)^i [\k]$ does not converge.
\end{Remark} 

Having shown that $\P^- \1_\uk$ is well-defined it is then easy to see that $\P^- \1_\uk$ is idempotent, meaning that $\P^- \P^- \1_\uk \cong \P^- \1_\uk$. The main result of \cite{C1} showed using an instance of skew Howe duality that $\P^-$ can be used to categorify all the clasps. The inspiration of using infinite twists to categorify clasps goes back to \cite{Roz} who categorified Jonez-Wenzl projectors within Bar-Natan's graphical formulation of Khovanov homology.

\section{The category $\K_n$}\label{sec:Kn}

\subsection{Categories and functors}

We now define a 2-category $\K_n$ with an $(\sl_n,\t)$ action. 

For $k \in \N$ consider the affine space $\A^k := \Spec \k[x_1, \dots, x_k]$ where $\deg(x_\ell)=2$ (the grading is equivalent to endowing $\A^k$ with a $\k^\times$ action). The quotient $\A_k := \A^k/S_k$ by the symmetric group $S_k$ in $k$ letters is isomorphic to $\Spec \k[e_1, \dots, e_k]$ where $e_\ell$ are the elementary symmetric functions and $\deg(e_\ell) = 2\ell$. For a sequence $\uk$ we write $\A_\uk := \A_{k_1} \times \dots \times \A_{k_n}$. Finally, we will denote by $D(\A_\uk)$ the derived category of $\k^\times$-equivariant quasi-coherent sheaves on $\A_\uk$. We will denote by $\{\cdot\}$ a shift in the grading induced by the $\k^\times$ action. In particular, this means that multiplication by $e_\ell$ induces a map $\O_{\A_{\uk}} \rightarrow \O_{\A_{\uk}} \{ 2 \ell \}$ since $e_\ell$ has degree $2 \ell$. This is the same convention as in earlier papers such as \cite{CK1}.

For $n \in \N$ the 2-category $\K_n$ is defined as follows. The objects are the categories $D(\A_\uk)$. The 1-morphisms are all kernels on products $\A_\uk \times \A_{\uk'}$ (with composition given by the convolution product $\star$) and the 2-morphisms are morphisms between kernels. The grading shift $\la 1 \ra$ is by definition $\{1\}$. 

Note that for $a,b \in \N$ there exists a natural projection map $\pi: \A_{a,b} \rightarrow \A_{a+b}$. This map is finite of degree $\binom{a+b}{a}$. More generally, we can consider correspondences such as
\begin{equation}\label{eq:corr}
\xymatrix{
& \A_{(\dots, k_i-r,r,k_{i+1},\dots)} \ar[ld]^{\pi_1} \ar[rd]_{\pi_2} & \\
\A_\uk & & \A_{\uk+r\alpha_i}
}
\end{equation}
where $\alpha_i = (0,\dots,-1,1,\dots,0)$ with a $-1$ in position $i$. We then define the following data. 
\begin{itemize}
\item 1-morphisms 
\begin{align*}
& \sE_i \1_\uk := \O_{\A_{(\dots,k_i-1,1,k_{i+1},\dots)}} \{k_i-1\} \in D(\A_\uk \times \A_{\uk + \alpha_i}) \\
& \1_\uk \sF_i := \O_{\A_{(\dots,k_i-1,1,k_{i+1},\dots)}} \{k_{i+1}\} \in D(\A_{\uk + \alpha_i} \times \A_\uk)
\end{align*}
where we embed $\A_{(\dots,k_i-1,1,k_{i+1},\dots)} \subset \A_\uk \times \A_{\uk+\alpha_i}$ using $\pi_1$ and $\pi_2$ from (\ref{eq:corr}) (taking $r=1$ in this case). 
\item A $\k$-linear map $\theta: \spn\{\alpha_i: i \in I\} \rightarrow \End^2(\1_\uk)$ where the image of $\alpha_i$ is given by mutiplication by $e_1^{(i)} - e_1^{(i+1)}$ where $e_1^{(i)}, e_2^{(i)}, \dots, e_{k_i}^{(i)}$ are the elementary generators of the factor $\A_{k_i}$ inside $\A_{\uk}$. 
\end{itemize}

\begin{Remark}
Although we use derived categories of quasi-coherent sheaves, we could restrict everything to abelian categories of coherent sheaves. This is because all the morphisms involved are flat and finite. However, it is natural to work with these larger categories because later we will apply Hochschild cohomology. 
\end{Remark}

\begin{Theorem}\label{thm:action}
The data above defines an $(\sl_n,\t)$ action on $\K_n$. 
\end{Theorem}
\begin{proof}
The fact that relations of an $(\sl_n,\t)$ action are satisfied is not difficult to prove and essentially follows from \cite[Section 6]{KL}. The fact that $\t$ satisfies relation (\ref{co:theta}) comes down to the following elementary fact. Consider $\k[x,y]$ as a $\k[x,y]^{S_2} \cong \k[e_1,e_2]$ bimodule where $S_2$ acts by switching $x$ and $y$ while (following our notation above) $e_1=x+y,e_2=xy$. Then, as a bimodule, 
$$\k[x,y] \cong \k[e_1,e_2] \oplus \k[e_1,e_2] \{2\}$$
and multiplication by $x$ (or $y$) induces an endomorphism of $\k[x,y]$ which is an isomorphism between the summands $\k[e_1,e_2] \{2\}$ on either side. 

Perhaps one thing to note is that our choices of shifts in defining the $\sE_i$s and $\sF_i$s differ slightly from \cite{KL}. However, the specific choice of shifts is not so important and is mainly determined by the fact that the canonical bundle of $\A_\uk$ is $\omega_{\A_\uk} \cong \O_{\A_\uk} \{d_\uk\}$ where $d_\uk = - \sum_\ell k_\ell(k_\ell+1)$. 
\end{proof}

It is not hard to show that the divided powers $\E_i^{(r)} \1_\uk$ and $\1_\uk \F_i^{(r)}$ are given by kernels
\begin{align*}
& \sE_i^{(r)} \1_\uk := \O_{\A_{(\dots,k_i-r,r,k_{i+1},\dots)}} \{ r(k_i-r) \} \in D(\A_\uk \times \A_{\uk + r \alpha_i}) \\
& \1_\uk \sF_i^{(r)} := \O_{\A_{(\dots,k_i-r,r,k_{i+1},\dots)}} \{ rk_{i+1} \} \in D(\A_{\uk + r \alpha_i} \times \A_\uk)
\end{align*}
where again we embed $\A_{(\dots,k_i-r,r,k_{i+1},\dots)}$ using (\ref{eq:corr}) (we will not use this fact). 

\begin{Remark}
Note that there are three different gradings that show up. First there is $\la 1 \ra = \{1\}$ which corresponds to the grading induced by the $\k^\times$ action. Second there is the cohomological grading $[1]$ in $\Kom^-_*(\K_n)$. Third there is the cohomological grading $[[1]]$ which is internal to $D(\A_{\uk})$. This last grading only shows up when we apply the trace 2-functors described in Section \ref{sec:trace}. 
\end{Remark}

\subsection{The braid group action}\label{sec:braid2}

Following Section \ref{sec:braid} we define the braid group generators $\sT_i \1_\uk \in \Kom(\K_n)$ as 
\begin{align*}
\left[ \dots \rightarrow \sE_i^{(-\la \uk, \alpha_i \ra+2)} \star \sF_i^{(2)} \{ -2 \} \rightarrow \sE_i^{(-\la \uk, \alpha_i \ra+1)} \star \sF_i \{-1\} \rightarrow \sE_i^{(-\la \uk, \alpha_i \ra)} \right] \1_\uk & \text{ if } \la \uk, \alpha_i \ra \le 0 \\
\left[\dots \rightarrow \sF_i^{(\la \uk, \alpha_i \ra+2)} \star \sE_i^{(2)} \{-2\} \rightarrow \sF_i^{(\la \uk, \alpha_i \ra+1)} \star \sE_i \{-1\} \rightarrow \sF_i^{(\la \uk, \alpha_i \ra)} \right] \1_\uk & \text{ if } \la \uk, \alpha_i \ra \ge 0
\end{align*}
We also get the corresponding projectors $\sP^- \1_\uk \in \Kom^-_*(\K_n)$. 

Following the construction in \cite[Section 7.1]{C1} it is useful to also define the elements 
$$\sT'_i \1_\uk := \begin{cases} 
\sT_i \1_\uk [-k_{i+1}] \{k_{i+1}+k_ik_{i+1}\} & \text{ if } \la \uk, \alpha_i \ra \le 0 \\ 
\sT_i \1_\uk [-k_i] \{k_i+k_ik_{i+1}\} & \text{ if } \la \uk, \alpha_i \ra \ge 0. \end{cases}$$
Notice that in contrast to \cite[Section 7.1]{C1} we have an extra shift of $\{k_ik_{i+1}\}$. These $\sT'_i$ also generate a braid group action but are better behaved with respect to $\sE$'s and $\sF$'s since, using \cite[Corollary 7.3]{C1} and \cite[Corollary 4.6]{C1}, we have
\begin{align}
\label{rel:EFT} & \sT'_i \star \sT'_j \star \sE_i \cong \sE_j \star \sT'_i \star \sT'_j \ \ \text{ and } \ \ \sT'_i \star \sT'_j \star \sF_i \cong \sF_j \star \sT'_i \star \sT'_j \ \ \text{ if } \ \  |i-j|=1  \\
\label{rel:ET} & \sT'_i \star \sE_i \cong \sF_i \star \sT'_i  \ \ \text{ and } \ \ \sT'_i \star \sF_i \cong \sE_i \star \sT'_i.
\end{align}

\subsection{Trace 2-functors}\label{sec:trace}

For any $\ell \in \N$ we now define a 2-functor $\Psi_\ell: \K_n \rightarrow \K_{n+1}$. This functor should be thought of as adding a strand labeled $\ell$. 

At the level of objects $\Psi_\ell$ takes $D(\A_\uk)$ to $D(\A_{\uk,\ell})$. Given a 1-morphism $\sM \in D(\A_\uk \times \A_{\uk'})$ we define 
\begin{equation}\label{eq:psi}
\Psi_\ell(\sM) := \Delta_* \pi^* (\sM) \in D(\A_{\uk,\ell} \times \A_{\uk',\ell})
\end{equation}
where $\pi^*$ and $\Delta_*$ are pullback and pushforward with respect to the natural projection and diagonal maps 
$$\pi: \A_\uk \times \A_\ell \times \A_{\uk'} \rightarrow \A_\uk \times \A_{\uk'} \text{ and } \Delta: \A_\uk \times \A_\ell \times \A_{\uk'}  \rightarrow (\A_\uk \times \A_\ell) \times (\A_{\uk'} \times \A_\ell).$$
Given a 2-morphism $f: \sM \rightarrow \sM'$ we define $\Psi_\ell(f) := \Delta_* \pi^*(f)$. Using Corollary \ref{cor:psi} (see the Appendix) this defines a 2-functor $\Psi_\ell: \K_n \rightarrow \K_{n+1}$. It is not difficult to see that $\Psi_\ell(\sE_i \1_\uk) \cong \sE_i \1_{\uk,\ell}$ and $\Psi_\ell(\1_\uk \sF_i) \cong \1_{\uk,\ell} \sF_i$. 

We can likewise define a 2-functor $\Psi'_\ell: \K_{n+1} \rightarrow \K_n$. On objects it takes $D(\A_{\uk,\ell})$ to $D(\A_\uk)$. All other objects, meaning $D(\A_{\uk,\ell'})$ where $\ell \ne \ell'$, are mapped to zero. On 1-morphisms it acts by 
$$D(\A_{\uk,\ell} \times \A_{\uk',\ell}) \ni \sN \mapsto \Psi'_\ell(\sN) := \pi_* \Delta^*(\sN) \in D(\A_\uk \times \A_{\uk'}).$$
By Propositions \ref{prop:trace1} we also have 
\begin{equation}
\label{eq:trace1} \Psi'_\ell(\sN \star \Psi_\ell(\sM)) \cong \Psi'_\ell(\sN) \star \sM 
\end{equation}
where $\sM \in D(\A_\uk \times \A_{\uk'})$ and $\sN \in D(\A_{\uk,\ell} \times \A_{\uk',\ell})$. 

If we denote by $\bullet = \Spec \k$ then $D(\bullet)$ is the category of complexes of (possibly infinite-dimensional) graded vector spaces. For any $\uk$ we define 
$$\tau: D(\A_\uk \times \A_\uk) \rightarrow D(\bullet) \ \ \text{ where } \ \ \sM \mapsto \Psi'_{k_1} \circ \dots \circ \Psi'_{k_n}(\sM).$$
Note that this is just the Hochschild homology $HH_*(\sM)$ of $\sM$. 

\section{Link invariants}\label{sec:linkinvariants}

Consider an oriented link $L$ whose components are coloured by partitions. For now we assume that each such partition has only one part, meaning it is of the form $(k)$ for some $k \in \N$. Such a link can be given as the closure $\hbeta$ of a coloured braid $\beta$, where we visualize the strands of this braid vertically with the top and bottom labeled by the same sequence $\uk$. 

To a positive crossing exchanging strands $i$ and $i+1$ (i.e. the strand starting at $i$ crosses over the one starting at $i+1$) we associate the 1-morphism
$$\sT'_i \in \Kom^-_*(D(\A_\uk \times \A_{s_i \cdot \uk}))$$
as defined in Section \ref{sec:braid2}, where $s_i$ acts on $\uk$ by switching $k_i$ and $k_{i+1}$. Composing these 1-morphisms gives a complex $\sT'_\beta \in \Kom^-_*(D(\A_\uk \times \A_{\uk'}))$. The invariant associated to the closure $\hbeta$ of the braid is then $\tau(\sT'_\beta) \in \Kom^-_*(D(\bullet))$. 

To deal with partitions with more than one parts we cable strands together and use the projector $\sP^-$. More precisely, given a strand labeled by a partition $k_i^{(\cdot)} = (k_i^{(1)} \le \dots \le k_i^{(p)})$, we replace it with $p$ strands labeled $k_i^{(1)}, \dots, k_i^{(p)}$ together with the projector $\sP^- \1_{k_i^{(\cdot)}}$ on these strands. 

\begin{Theorem}\label{thm:1}
Suppose $L = \widehat{\beta}$ where $\beta$ is a braid whose strands are coloured by partitions. Then, up to an overall grading shift, $\sH(L) := \tau(\sT'_\beta) \in \Kom^-_*(D(\bullet))$ defines a triply graded link invariant.
\end{Theorem}
\begin{Remark}
In order to obtain a homology which is invariant on the nose (not just up to shifts) one needs to shift the functor $\Psi'_\ell$ by $[\frac{\ell}{2}][[-\frac{\ell}{2}]]$ and the definition of a $\sT$ switching two strands labeled $\ell$ by $[\frac{\ell}{2}] [[-\frac{\ell}{2}]]$. 
\end{Remark}

Before we can prove Theorem \ref{thm:1} we need the following Lemma. 

\begin{Lemma}\label{lem:technical}
For $\sT_1 \in \Kom^*_-(D(\A_{1,1} \times \A_{1,1}))$ we have 
$$\Psi_1'(\sT_1) \cong \O_\Delta [[1]]\{-2\} \ \ \text{ and } \ \ \Psi_1'(\sT_1^{-1}) \cong \O_\Delta \{2\}[-1]$$ 
inside $\Kom^*_-(D(\A_1 \times \A_1))$.
\end{Lemma}
\begin{Remark}\label{rem:1}
The key to Lemma \ref{lem:technical} is the exact triangle $\O_S \rightarrow \O_\Delta \rightarrow \O_T[[1]]\{-2\}$ where $S$ and $T$ are the loci inside $\A_{\uk,1,1} \times \A_{\uk,1,1}$ given by (\ref{eq:4}) on the last two strands. The argument in the proof shows in fact that for $\sP \in \Kom(D(\A_{\uk,1} \times \A_{\uk,1}))$ we have an isomorphism 
\begin{equation}\label{eq:6}
\Psi_1'((\O_S \rightarrow \O_\Delta) \star \Psi_1(\sP)) \xrightarrow{\sim} \Psi'_1(\O_T [[1]]\{-2\} \star \Psi_1(\sP))
\end{equation}
inside $\Kom(D(\A_{\uk,1} \times \A_{\uk,1}))$. 
\end{Remark}
\begin{proof}
On $\A_{1,1} \times \A_{1,1}$ consider the following subvarieties:
\begin{equation}\label{eq:4}
\Delta := \{(x,y,x,y)\}, \ \ T := \{(x,y,y,x)\} \ \ \text{ and } \ \ S := T \cup \Delta.
\end{equation}
Then $\sT_1 \cong [\O_S \rightarrow \O_\Delta]$ and $\sT_1^{-1} \cong [\O_\Delta \rightarrow \O_S \{2\}]$ where in both cases $\O_\Delta$ is in cohomological degree zero. The result will follow if we can show that 
\begin{align*}
[\Psi'_1(\O_S) \rightarrow \Psi'_1(\O_\Delta)] &\cong [0 \rightarrow \O_\Delta [[1]]\{-2\}] \\
[\Psi'_1(\O_\Delta) \rightarrow \Psi'_1(\O_S\{2\} )] &\cong [0 \rightarrow \O_\Delta\{2\}]
\end{align*}
in the homotopy category $\Kom^-_*(D(\A_1 \times \A_1))$. 

We will prove the first assertion (the second follows similarly). Note that $S \cap T \subset T$ is the divisor cut out by $x=y$. Thus, $\O_T(- S \cap T) \cong \O_T \{-2\}$ and we have the exact triangle $\O_T \{-2\} \rightarrow \O_S \rightarrow \O_\Delta$. Recall that $\Psi'_1(\cdot) = \pi_* \Delta^* (\cdot)$  where $\pi$ and $\Delta$ are the natural maps
$$\A_1 \times \A_1 \xleftarrow{\pi} \A_1 \times \A_1 \times \A_1 \xrightarrow{\Delta} \A_{1,1} \times \A_{1,1}.$$
Now, $\O_\Delta \in D(\A_{1,1} \times \A_{1,1})$ has a resolution 
$$\O_{\A_{1,1} \times \A_{1,1}} \{-2\} \xrightarrow{\cdot (y_1-y_2)} \O_{\A_{1,1} \times \A_{1,1}} \rightarrow \O_\Delta$$
which means that $\Delta^* \O_\Delta \cong \O_{\Delta'} \oplus \O_{\Delta'}[[1]]\{-2\}$ where $\Delta' \subset \A_1 \times \A_1 \times \A_1$ is the locus $(x,y,x)$. Moreover, $\pi_*(\O_{\Delta'}) \cong \O_\Delta \otimes_\k \k[y] \in D(\A_1 \times \A_1)$. Hence
\begin{equation}\label{eq:A}
\Psi'_1(\O_\Delta) = \pi_* \Delta^* \O_\Delta \cong \pi_*(\O_{\Delta'} \oplus \O_{\Delta'} [[1]]\{-2\}) \cong (\O_\Delta \oplus \O_\Delta [[1]]\{-2\}) \otimes_\k \k[y].
\end{equation}
On the other hand, $\Delta^* \O_T \cong \O_{\{(x,x,x)\}}$ which means that
$$\Psi'_1(\O_T \{-2\}) = \pi_* \Delta^* \O_T \{-2\} \cong \pi_* \O_{\{(x,x,x)\}} \{-2\} \cong \O_\Delta \{-2\}.$$
Thus, the exact triangle $\Psi'_1(\O_S) \rightarrow \Psi'_1(\O_\Delta) \rightarrow \Psi'_1(\O_T \{-2\} [1])$ becomes
\begin{equation*}
\Psi'_1(\O_S) \xrightarrow{f} \O_\Delta \otimes_\k \k[y] \oplus \O_\Delta [[1]]\{-2\} \otimes_\k \k[y] \rightarrow \O_\Delta [[1]] \{-2\}. 
\end{equation*}
Now $\sH^0(\Psi'_1(\O_S)) \cong \pi_* L^0 \Delta^* \O_S \cong \pi_* \O_{\Delta'} \cong \O_\Delta \otimes_\k \k[y]$ and $f$ induces an isomorphism in this degree. Thus, from the long exact sequence we get 
\begin{equation}\label{eq:B}
\sH^*(\Psi'_1(\O_S)) = 
\begin{cases}
\O_\Delta \otimes_\k \k[y] \ \ &\text{ if } *=0, \\
\O_\Delta \{-4\} \otimes_\k \k[y] \ \ &\text{ if } *=-1, \\
0 \ \ &\text{ otherwise. }
\end{cases}
\end{equation}
It is easy to see that on $\A_1 \times \A_1$ we have $\End^2(\O_\Delta \{\cdot\} \otimes_\k \k[y]) = 0$ so 
$$\Psi'_1(\O_S) \cong (\O_\Delta \oplus \O_\Delta [[1]]\{-4\}) \otimes_\k \k[y].$$
Hence, using a version of the Gaussian Elimination Lemma \cite[Lemma 3.2]{C1}, we combine (\ref{eq:A}) and (\ref{eq:B}) to obtain
\begin{align}\label{eq:C}
[\Psi'_1(\O_S) \rightarrow \Psi'_1(\O_\Delta)] \cong [0 \rightarrow \O_\Delta [[1]]\{-2\}].
\end{align}
\end{proof}

\begin{proof}[Proof of Theorem \ref{thm:1}]
We already know that $\beta \mapsto \sT'_\beta$ satisfies the braid relations. It remains to check that $\tau(\sT'_{\beta_1} \star \sT'_{\beta_2}) \cong \tau(\sT'_{\beta_2} \star \sT'_{\beta_1})$ and the Markov move (stabilization). 

The first relation is a standard property of Hochschild homology (in fact the more general trace property $\tau(\sA \star \sB) \cong \tau(\sB \star \sA)$ holds for any kernels $\sA, \sB$).

To prove the Markov move first note that since projectors $\sP^-$ move freely through crossings it suffices to prove the Markov move when the extra strand is coloured by a partition $(\ell)$. In this case, for any $\sP \in \Kom^-_*(D(\A_{\uk} \times \A_{\uk})$ we claim that 
\begin{equation}\label{eq:tr1}
\tau(\sT'_n \star \Psi_\ell(\sP)) \cong \tau(\sP)[-\ell] [[\ell]] \ \ \text{ and } \ \ \tau((\sT'_n)^{-1} \star \Psi_\ell(\sP)) \cong \tau(\sP).  
\end{equation}
We prove the isomorphism on the left by induction on $\ell$ (the right one is similar). For $\sP \in \Kom^-_*(D(\A_{\uk} \times \A_{\uk}))$ we have the following algebraic computation
\begin{align}
\bigoplus_{[\ell]} \tau(\sT'_n \star \Psi_\ell(\sP)) 
&\cong \bigoplus_{[\ell]} \tau(\sT'_{n+1} \star \sT'_n \star \sT'_{n+2} \star \sT'_{n+1} \star (\Psi_0 \circ \Psi_\ell \circ \Psi_0)(\sP)) \\
\label{line:1} &\cong \tau(\sT'_{n+1} \star \sT'_n \star \sT'_{n+2} \star \sT'_{n+1} \star \sF_{n+2} \star \sE_{n+2} \star (\Psi_0 \circ \Psi_\ell \circ \Psi_0)(\sP)) \\
\label{line:2} &\cong \tau(\sF_n \star \sT'_{n+1} \star \sT'_n \star \sT'_{n+2} \star \sT'_{n+1} \star \sE_{n+2} \star (\Psi_0 \circ \Psi_\ell \circ \Psi_0)(\sP)) \\
\label{line:3} &\cong \tau(\sE_{n+2} \star \sF_n \star \sT'_{n+1} \star \sT'_n \star \sT'_{n+2} \star \sT'_{n+1} \star (\Psi_1 \circ \Psi_{\ell-1} \circ \Psi_0)(\sP)) \\
\label{line:4} &\cong \tau(\sF_n \star \sT'_{n+1} \star \sT'_n \star \sT'_{n+2} \star \sT'_{n+1} \star \sE_n \star (\Psi_1 \circ \Psi_{\ell-1} \circ \Psi_0)(\sP)) \\
\label{line:5} &\cong \tau(\sF_n \star \sT'_{n+1} \star \sT'_n \star \sT'_{n+1} \star \sE_n \star (\Psi_{\ell-1} \circ \Psi_0)(\sP)) [-1] [[1]] \\
\label{line:6} &\cong \tau(\sF_n \star \sT'_n \star \sT'_{n+1} \star \sT'_n \star \sE_n \star (\Psi_{\ell-1} \circ \Psi_0)(\sP)) [-1] [[1]] \\
\label{line:7} &\cong \tau(\sF_n \star \sT'_n \star \sT'_n \star \sE_n \star \Psi_0(\sP)) [-\ell] [[\ell]] \\
\label{line:8} &\cong \tau(\sT'_n \star \sE_n \star \sF_n \star \sT'_n \star \Psi_0(\sP)) [-\ell] [[\ell]] \\
\label{line:9} &\cong \bigoplus_{[\ell]} \tau(\sP) [-\ell] [[\ell]].
\end{align}
Here we added two strands labeled $0$ to obtain the first isomorphism, used (\ref{rel:EFT}) twice to obtain (\ref{line:2}), used the Markov relation to obtain (\ref{line:3}), used that 
\begin{align*}
\sE_{n+2} \star \sF_n \star \sT'_{n+1} \star \sT'_n \star \sT'_{n+2} \star \sT'_{n+1} 
&\cong \sF_n \star \sE_{n+2} \star \sT'_{n+1} \star \sT'_{n+2} \star \sT'_n \star \sT'_{n+1} \\
&\cong \sF_n \star \sT'_{n+1} \star \sT'_{n+2} \star \sT'_n \star \sT'_{n+1} \star \sE_n
\end{align*}
to get (\ref{line:4}), used (\ref{eq:tr1}) with $\Psi_1$ to obtain (\ref{line:5}) and with $\Psi_{\ell-1}$ to obtain (\ref{line:7}), applied (\ref{rel:ET}) twice to obtain (\ref{line:9}) and used that $(\sT'_i)^2$ is the identity if one of the strands it acts on is labeled $0$ to get (\ref{line:9}).

Thus, (\ref{eq:tr1}) follows by induction if we can prove the base case $\ell=1$. In this case we have
$$\tau(\sT'_n \star \Psi_1(\sP)) \cong \tau(\Psi'_1(\sT'_n \star \Psi_1(\sP))) \cong \tau(\Psi'_1(\sT'_n) \star \sP)$$
so it suffices to show that $\Psi'_1(\sT'_n) \cong \O_\Delta [-1][[1]]$. This follows from Lemma \ref{lem:technical} (since $\sT_n' = \sT_n [-1]\{2\}$ in this case). 

\begin{Remark}
For those familiar with webs, see for instance \cite{CKM}, the algebraic computation above can be summarized as follows. First break up the strand labeled $\ell$ and then use that ``trivalent vertices'' move naturally through crossings together with the Markov move. Figure \ref{fig:2} illustrates this procedure where the box denotes an arbitrary braid (we simplify by omitting the closure of each diagram). 

\begin{figure}[H]
\begin{center}
\puteps[0.7]{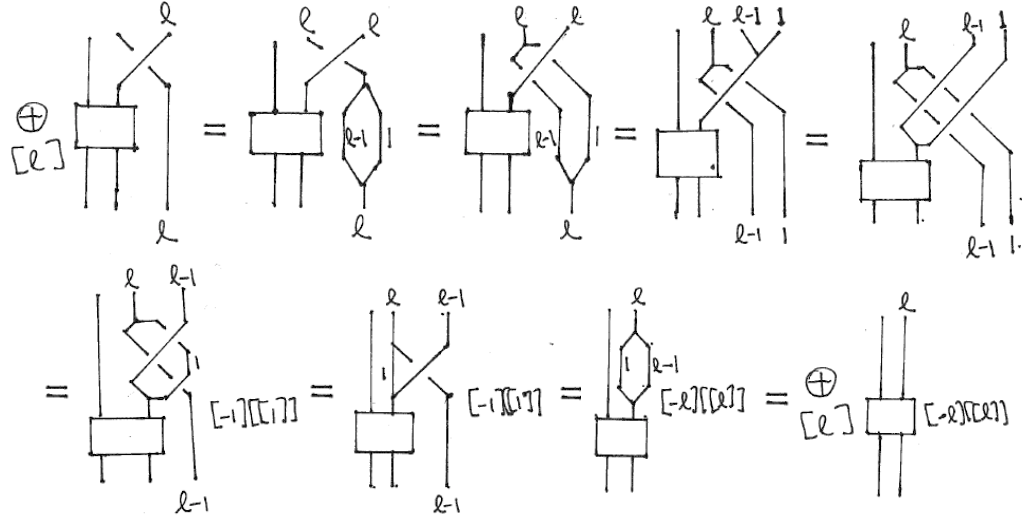}
\end{center}
\caption{The Markov move involving a strand labeled $\ell$. }\label{fig:2}
\end{figure}
\end{Remark}
\end{proof}

\section{K-theory}

Recall that to a link $L$ whose strands are coloured by partitions one can associate the coloured HOMFLY polynomial $P_L(q,a) \in \k(q,a)$. We now explain why the invariant from Theorem \ref{thm:1} categorifies the coloured HOMFLY polynomial. This is normalized so that if $L = \bigcirc_{(k)}$ (the unknot labeled by $k$) then 
\begin{equation}\label{eq:D}
P_L(q,a) = \prod_{\ell = 1}^k \frac{aq^{-\ell+1} - a^{-1}q^{\ell-1}}{q^{-\ell} - q^{\ell}}.
\end{equation}
\begin{Remark}
Note that for notational convenience we use the transpose notation, meaning that what we call $\l$ would normally be the transpose partition. For example, our partition $(k)$ would be instead $(1^k)$ (and vice versa). 
\end{Remark}
On the other hand, we can consider the Poincar\'e polynomial $\sP_L(q,a,t)$ of $\sH(L)$ from Theorem \ref{thm:1}. Here the shifts $\{1\},[[1]]$ and $[1]$ are kept track of by formal variables $q,-a^2,t$ respectively. 

\begin{Proposition} 
For a coloured link $L$, the invariants $P_L(q,a)$ and $\sP_L(q,aq,-1)$ agree, up to an explicit factor $a^mq^n$.
\end{Proposition}
\begin{proof}
In the rest of the proof we will ignore extra factors $a^mq^n$. Let us first suppose that $L$ is the closure of a braid $\beta$ coloured by partitions $(k)$ with only one part. One can compute $P_L(q,a)$ from $\beta$ by applying a trace. Moreover, as explained (for instance) in \cite[Section 6]{CKM} one can break down the crossings in $L$ into web diagrams since the crossing element is a linear combination of webs. This reduces the evaluation of $P_L(q,a)$ to evaluating this trace on webs diagrams.  

As usual one views the trace of a web diagram as the closure on that diagram on the annulus. The algebra of webs on the annulus is generated (as an algebra) by unknots labeled by one part partitions (where multiplication is given by glueing one annulus inside the other). This reduces the computation of $P_L(q,a)$ to the case $L = \bigcirc_{(k)}$ (which is described in (\ref{eq:D})). 

Similarly, the evaluation of $\sP_L(q,a,-1)$ can be reduced to the case $L = \bigcirc_{(k)}$. This case is computed in Section \ref{sec:example1} and agrees with (\ref{eq:D}) once you replace $a$ with $aq$. This completes the proof when $L$ contains partitions with only one part. 

To deal with arbitrary partitions we will show that $P_L(q,q^N) = \sP_L(q,q^{N+1},-1)$ for all $N > 0$ (i.e. the specializations $a=q^N$ for all $N > 0$). Note that $P_L(q,q^N)$ recovers the corresponding $SL(N)$ Reshetikhin--Turaev (RT) invariant and that we know $P_L(q,q^N) = \sP_L(q,q^{N+1},-1)$ are equal if $L$ is coloured by one part partitions. On the other hand, in \cite{C1} we showed that, when evaluating RT invariants, the projectors (clasps) for arbitrary partitions can be constructed as infinite twists. Since this construction only uses the braiding group action it follows that $P_L(q,q^N) = \sP_L(q,q^{N+1},-1)$ holds for any $L$. 
\end{proof}

\section{Some differentials}\label{sec:differentials}

To simplify notation we will omit the $\{\cdot\}$ grading in this section. We also fix $N > 0$. Note that 
$$HH^1(\A^k) = \Ext_{\A^k \times \A^k}^1(\Delta_* \O_{\A^k}, \Delta_* \O_{\A^k}) \cong \oplus_i \k[x_1,\dots,x_k] \partial_{x_i}$$
so that 
$$\gamma_{1^k} := \sum_i x_i^N \partial_{x_i} \in \Hom_{\A^k \times \A^k}(\Delta_* \O_{\A^k}, \Delta_* \O_{\A^k} [1]).$$
Since this element is $S_k$-invariant it descends to $HH^1(\k[e_1,\dots,e_k])$. We denote by 
$$\gamma_{\uk} \in \Hom_{\A_\uk \times \A_\uk}(\Delta_* \O_{\A_\uk}, \Delta_* \O_{\A_\uk} [1])$$ 
the corresponding element obtained by descending $\gamma_{1^{k_1}} \otimes \dots \otimes \gamma_{1^{k_n}}$. 

Now, given a braid $\beta$ with endpoints marked $\uk$, we have
$$\tau(\sT'_\beta) = HH_*(\sT'_\beta) \cong \Ext^*_{\A_\uk \times \A_\uk}(\Delta_* \omega_{\A_\uk}^{-1}[[-\dim \A_\uk]], \sT'_\beta)$$
where $\omega_{\A_\uk}$ denotes the canonical bundle. Thus, we have an action of $HH^1(\A_\uk)$ coming from precomposing on the left with 
$$HH^1(\A_\uk) \cong \Hom_{\A_\uk \times \A_\uk}(\Delta_* \omega_{\A_\uk}^{-1}, \Delta_* \omega_{\A_\uk}^{-1} [1]).$$ 
We denote by $d_N$ the action of $\gamma_\uk$. Note that $d_N^2=0$ since $\gamma_\uk$ belongs to $HH^1$. Moreover, $d_N$ commutes with the differential $d$ used in the definition of the complex $\sT'_\beta$ because of associativity of composition. Thus, we get a bicomplex with differentials $d$ and $d_N$. 

\begin{Theorem}\label{thm:2}
Suppose $L = \widehat{\beta}$ where $\beta$ is a coloured braid. If we denote by $\sH_N(L)$ the cohomology of $\tau(\sT'_\beta)$ equipped with the total differential $d+d_N$ then, up to an overall grading shift, $\sH_N(L)$ defines a doubly graded link invariant. 
\end{Theorem}

In the remainder of this section we prove this result. Sometimes we will write $\sH_N(\sT'_\beta)$ instead of $\sH_N(L)$ where $L = \widehat{\beta}$.

If $\beta$ and $\beta'$ are equivalent braids, then $\sT'_\beta$ is homotopic to $\sT'_{\beta'}$, which means that $\sH_N(\sT'_\beta) \cong \sH_N(\sT'_{\beta'})$. Next, to prove invariance under conjugation, we must show that $\sH_N(\sT'_{\beta_1} \star \sT'_{\beta_2}) \cong \sH_N(\sT'_{\beta_2} \star \sT'_{\beta_1})$ for any braids $\beta_1$ and $\beta_2$. This follows as in the proof of Theorem \ref{thm:1} together with the fact that for any braid $\beta$ we have
$$II \gamma_\uk = \gamma_{\uk'} II \in \Hom(\Delta_* \O_{\A_{\uk'}} \star \sT'_\beta \star \Delta_* \O_{\A_\uk}, \Delta_* \O_{\A_{\uk'}} \star \sT'_\beta \star \Delta_* \O_{\A_\uk} [1])$$
where $\uk$ and $\uk'$ label the bottom and top strands of $\beta$ (this equality follows from Lemma \ref{lem:2} below). 
\begin{Remark}
Here we use the convention mentioned at the end of Section \ref{sec:notation}. For instance, $II \gamma_\uk$ denotes the map induced by the identity on the first two factors of $\Delta_* \O_{\A_{\uk'}} \star \sT'_\beta \star \Delta_* \O_{\A_\uk}$ and by $\gamma_\uk$ on the last (right) one. 
\end{Remark}

\begin{Lemma}\label{lem:2}
Consider $\sE \in D(\A_{i-1,j+1} \times \A_{i,j})$. Then 
$$II \gamma_{i,j} = \gamma_{i-1,j+1} II \in \Hom(\Delta_* \O_{\A_{i-1,j+1}} \star \sE \star \Delta_* \O_{\A_{i,j}}, \Delta_* \O_{\A_{i-1,j+1}} \star \sE \star \Delta_* \O_{\A_{i,j}} [1])$$
and likewise if we replace $\sE$ with $\sF$. 
\end{Lemma}
\begin{proof}
$\sE$ is the kernel inducing the correspondence
$$\A_i \times \A_j \xleftarrow{\pi_1} \A_{i-1} \times \A^1 \times \A_j \xrightarrow{\pi_2} \A_{i-1} \times \A_{j+1}.$$
On the other hand $II \gamma_{i,j}$ is the element obtained by descending $\gamma_{1^{i+j}}$ via the map 
$$\A^{i+j} \rightarrow \A_{i-1} \times \A^1 \times \A_j \xrightarrow{\pi_1} \A_i \times \A_j.$$ 
Likewise $\gamma_{i-1,j+1} II$ is the map obtained by descending $\gamma_{1^{i+j}}$ again via the map 
$$\A^{i+j} \rightarrow \A_{i-1} \times \A^1 \times \A_j \xrightarrow{\pi_2} \A_{i-1} \times \A_{j+1}.$$
The result follows. 
\end{proof}

Finally, we need invariance under the Markov move. As in the proof of Theorem \ref{thm:1} we can significantly reduce what we must show. First, since projectors pass through crossings we can assume each strand is coloured by a partition $(\ell)$ with only one part. By breaking up this strand into $\ell$ strands coloured by $1$ and using Lemma \ref{lem:2} we can further reduce to the case $\ell=1$. 

Using the homotopy equivalence from (\ref{eq:6}) we know that 
$$\sH_N(\Psi_1'(\sT'_n \star \Psi_1(\sP))) = \sH_N(\Psi_1'((\O_S \rightarrow \O_\Delta) \star \Psi_1(\sP))) \xrightarrow{\sim} \sH_{-N}(\Psi_1'(\O_T[[1]]\{-2\} \star \Psi_1(\sP)))$$
where $\sP \in \Kom(D(\A_{\uk,1} \times \A_{\uk,1}))$. Recall that we have the standard exact sequence
$$\O_T \{-2\} \rightarrow \O_S \rightarrow \O_\Delta$$
where $S,T$ are the varieties corresponding to the last two strands. Moreover, $\Psi_1'(\O_T \star \Psi_1(\sP)) \xrightarrow{\sim} \sP$ and Lemma \ref{lem:3} implies that $\sH_N(\sT'_n \star \Psi_1(\sP)) \cong \sH_N(\sP)$ (up to a grading shift). This completes the proof of Theorem \ref{thm:2}. 

\begin{Lemma}\label{lem:3}
For $\sP \in \Kom(D(\A_{\uk,1} \times \A_{\uk,1}))$ the following diagram commutes
$$\xymatrix{
\Ext^j_{\A_{\uk,1,1} \times \A_{\uk,1,1}}(\Delta_* S_{\A_{\uk,1,1}}^{-1}, \O_T \star \Psi_1(\sP)) \ar[r]^{\gamma_{\uk,1,1}} \ar[d]^{\varphi} & \Ext^{j+1}_{\A_{\uk,1,1} \times \A_{\uk,1,1}}(\Delta_* S_{\A_{\uk,1,1}}^{-1},\O_T \star \Psi_1(\sP)) \ar[d]^{\varphi'} \\
\Ext^j_{\A_{\uk,1} \times \A_{\uk,1}}(\Delta_* S_{\A_{\uk,1}}^{-1}, \sP) \ar[r]^{\gamma_{\uk,1}} & \Ext^{j+1}_{\A_{\uk,1} \times \A_{\uk,1}}(\Delta_* S_{\A_{\uk,1}}^{-1}, \sP) 
}$$
where $S^{-1}_X := \omega^{-1}_X [[-\dim X]]$ for a variety $X$ and where isomorphisms $\varphi$ and $\varphi'$ are induced by the isomorphism $\Psi_1'(\O_T \star \Psi_1(\sP)) \xrightarrow{\sim} \sP$. 
\end{Lemma}
\begin{proof}
As before, we will ignore shifts in $\{\cdot\}$. Note that the left adjoint of $\Psi'_1: \A_{\uk,1} \rightarrow \A_{\uk}$ is the functor 
$$(\Psi'_1)^L(\cdot) = \Delta_*(\pi^*(\cdot) \otimes p^* S_{\A^1}^{-1})$$
where $p: \A_\uk \times \A^1 \times \A_\uk \rightarrow \A^1$ is the projection. Now take 
$$\alpha \in \Ext^j_{\A_{\uk,1,1} \times \A_{\uk,1,1}}(\Delta_* S_{\A_{\uk,1,1}}^{-1}, \O_T \star \Psi_1(\sP))$$ 
and consider the following diagram 
\tiny
$$\xymatrix{
\Delta_* S^{-1}_{\A_{\uk,1}} [[-1]] \ar[rr]^{\gamma_{\uk,1}} \ar[d]^{\adj} & & \Delta_* S^{-1}_{\A_{\uk,1}} \ar[d]^{\adj} \ar@/^1pc/[drr]^{\varphi(\alpha)} & & \\
\Psi_1' (\Psi'_1)^L (\Delta_* S^{-1}_{\A_{\uk,1}}) [[-1]] \ar[rr]^{\Psi'_1 (\Psi'_1)^L (\gamma_{\uk,1})} & & \Psi_1' (\Psi'_1)^L (\Delta_* S^{-1}_{\A_{\uk,1}}) \ar[r]^-{\Psi'_1(\alpha)} & \Psi'_1(\O_T \star \Psi_1(\sP)) [[j]] \ar[r]^-{\sim} & \Psi_1'(\O_T) \star \sP [[j]] 
}$$
\normalsize
Note that $(\Psi'_1)^L(S^{-1}_{\A_{\uk,1}}) = S^{-1}_{\A_{\uk,1,1}}$ which explains how $\Psi'_1(\alpha)$ acts. The left square commutes since adjunction is a natural transformation. The triangle on the right commutes by the definition of $\varphi$. 

The composition along the top is the map $\alpha \mapsto \varphi(\alpha) \circ \gamma_{\uk,1}$. On the other hand, the composition along the bottom row is $\alpha \mapsto \varphi'(\alpha \circ (\Psi'_1)^L(\gamma_{\uk,1}))$. So it suffices to show that $\varphi'(\alpha \circ (\Psi'_1)^L(\gamma_{\uk,1})) = \varphi'(\alpha \circ \gamma_{\uk,1,1})$. The difference $\gamma_{\uk,1,1} - (\Psi'_1)^L(\gamma_{\uk,1})$ is equal to $x_{n+1}^N \p_{x_{n+1}}$ so it remains to show that $\varphi'(\alpha \circ x_{n+1}^N \p_{x_{n+1}}) = 0$. 

The map $\varphi'(\alpha \circ x_{n+1}^N \p_{x_{n+1}})$ is given by the composition 
\begin{align*}
\Delta_* S_{\A_{\uk,1}}^{-1} & \xrightarrow{\adj} \Psi_1' (\Psi'_1)^L (\Delta_* S^{-1}_{\A_{\uk,1}}) = \Psi_1' (\Delta_* S^{-1}_{\A_{\uk,1,1}}) \xrightarrow{\Psi'_1(x_{n+1}^N \p_{x_{n+1}})} \Psi_1' (\Delta_* S^{-1}_{\A_{\uk,1,1}} [[1]]) \\
& \xrightarrow{\Psi'_1(\alpha)} \Psi'_1(\O_T \star \Psi_1(\sP))[[j+1]] \xrightarrow{\sim} \sP [[j+1]].
\end{align*}
One can check that 
$$\Psi_1' (\Psi'_1)^L (\Delta_* S^{-1}_{\A_{\uk,1}}) \cong \Delta_* S^{-1}_{\A_{\uk,1}} \otimes_\k \k[x_{n+1}] \oplus \Delta_* S^{-1}_{\A_{\uk,1}} \otimes_\k \k[x_{n+1}] \p_{x_{n+1}} [[-1]].$$
Then the composition of the first two maps is given by 
\small
\begin{equation}\label{eq:7}
\begin{matrix}
 & & \Delta_* S^{-1}_{\A_{\uk,1}} \otimes_\k \k[x_{n+1}] \p_{x_{n+1}} [[-1]] & & \\ 
\Delta_* S^{-1}_{\A_{\uk,1}} & \xrightarrow{\id} & \bigoplus \Delta_* S^{-1}_{\A_{\uk,1}} \otimes_\k \k[x_{n+1}] & \xrightarrow{\cdot x_{n+1}^N \p_{x_{n+1}}} & \Delta_* S^{-1}_{\A_{\uk,1}} \otimes_\k \k[x_{n+1}] \p_{x_{n+1}} \bigoplus \\
 & & & & \Delta_* S^{-1}_{\A_{\uk,1}} \otimes_\k \k[x_{n+1}] [[1]] 
\end{matrix}
\end{equation}
\normalsize
where the missing arrows are all zero. 

On the other hand, to understand $\Psi'_1(\alpha)$, consider the isomorphisms
\begin{align*}
\Ext^j_{\A_{\uk,1,1} \times \A_{\uk,1,1}}(\Delta_* S^{-1}_{\A_{\uk,1,1}}, \O_T \star \Psi_1(\sP)) 
&\cong \Ext^j_{\A_{\uk,1,1} \times \A_{\uk,1,1}}((\Psi'_1)^L (\Delta_* S^{-1}_{\A_{\uk,1}}),  \O_T \star \Psi_1(\sP))  \\
&\cong \Ext^j_{\A_{\uk,1} \times \A_{\uk,1}}(\Delta_* S^{-1}_{\A_{\uk,1}}, \Psi_1'(\O_T \star \Psi_1(\sP))) \\
&\cong \Ext^j_{\A_{\uk,1} \times \A_{\uk,1}}(\Delta_* S^{-1}_{\A_{\uk,1}}, \sP).
\end{align*}
The image of $\beta: \Delta_* S^{-1}_{\A_{\uk,1}} \rightarrow \sP[[j]]$ under these isomorphisms is the composition 
$$(\Psi'_1)^L (\Delta_* S^{-1}_{\A_{\uk,1}}) \xrightarrow{(\Psi'_1)^L(\beta)} (\Psi'_1)^L(\sP)[[j]] = \Delta_* p^* (S^{-1}_{\A^1}) \star \Psi_1(\sP)[[j]] \xrightarrow{h}  (\O_T \star \Psi_1(\sP)) [[j]]$$
where $p: \A_{\uk,1,1} \rightarrow \A^1$ projects onto the last factor. Here $h$ is induced by the map $\Delta_* p^* (S^{-1}_{\A^1}) \rightarrow \O_T$ which comes from the standard exact sequence
$$\O_T \{-2\} \rightarrow \O_S \rightarrow \O_\Delta$$
after noting that $\Delta_* p^* (S^{-1}_{\A^1}) \cong \O_\Delta [[-1]] \{2\}$. Thus, we can assume that $\alpha$ is such a composition for some $\beta$. Applying $\Psi_1'$ we find that $\Psi'_1(\alpha)$ factors as 
\tiny
$$\begin{matrix}
\Delta_* S^{-1}_{\A_{\uk,1}} \otimes_\k \k[x_{n+1}] \p_{x_{n+1}} [[-1]] & \xrightarrow{\Psi_1' (\Psi'_1)^L(\beta)} & \sP \otimes_\k \k[x_{n+1}] \p_{x_{n+1}} [[j-1]] & & \\
\bigoplus \Delta_* S^{-1}_{\A_{\uk,1}} \otimes_\k \k[x_{n+1}] & \xrightarrow{\Psi_1' (\Psi'_1)^L(\beta)} & \bigoplus \sP \otimes_\k \k[x_{n+1}] [[j]] & \xrightarrow{\Psi_1'(h)} & \Psi_1'(\O_T \star \Psi_1(\sP)) [[j]]
\end{matrix}.$$
\normalsize
Finally, the composition of $\Psi_1'(h)$ with the isomorphism $\Psi_1'(\O_T \star \Psi_1(\sP))[[j]] \xrightarrow{\sim} \sP[[j]]$ gives a map which is zero on the summand $\sP \otimes_\k \k[x_{n+1}] \p_{x_{n+1}} [[j-1]]$ and the natural projection map $\sP \otimes_\k \k[x_{n+1}][[j]] \rightarrow \sP[[j]]$ on the second summand (which sends $x_{n+1}$ to zero). This fact can be traced back to the map 
$$\Delta_* \O_{\A_1} \otimes_\k \k[x_2] [[-1]] \oplus \Delta_* \O_{\A_1} \otimes_\k \k[x_2] = \Psi_1'(\Delta_* \O_{\A_{1,1}}) \rightarrow \Psi'_1(\O_T) \cong \Delta_* \O_{\A_1}$$
which, as we saw in the proof of Lemma \ref{lem:technical}, acts by zero on the first summand and by the natural projection map on the second summand. In conclusion, the composition 
$$\Psi_1' (\Psi'_1)^L (\Delta_* S^{-1}_{\A_{\uk,1}}) [[1]] \xrightarrow{\Psi'_1(\alpha)} \Psi'_1(\O_T \star \Psi_1(\sP))[[j+1]] \xrightarrow{\sim} \sP [[j+1]]$$
is isomorphic to the composition 
$$\begin{matrix}
\Delta_* S^{-1}_{\A_{\uk,1}} \otimes_\k \k[x_{n+1}] \partial_{x_{n+1}} & \xrightarrow{\Psi'_1 (\Psi'_1)^L(\beta)} & \sP \otimes_\k \k[x_{n+1}] \p_{x_{n+1}} [[j]] & & \\
\bigoplus \Delta_* S^{-1}_{\A_{\uk,1}} \otimes_\k \k[x_{n+1}] [[1]] & \xrightarrow{\Psi'_1 (\Psi'_1)^L(\beta)} & \bigoplus \sP \otimes_\k \k[x_{n+1}][[j+1]] & \xrightarrow{\pi} & \sP[[j+1]]
\end{matrix}$$
where $\pi$ is the natural projection map from the lower summand. The composition of this with (\ref{eq:7}) is clearly zero and hence $\varphi'(\alpha \circ x_{n+1}^N \p_{x_{n+1}}) = 0$. 
\end{proof}
\begin{Remark}
Note that in the proof above of Lemma \ref{lem:3} the observation we used is that the difference $\gamma_{\uk,1,1} - \Psi_1(\gamma_{\uk,1})$ is of the form $f \p_{x_{n+1}}$ for some function $f$ on $\A_{\uk,1,1}$. 
\end{Remark}

\section{Examples}\label{sec:examples}

For a partition $\uk = (k_1,\dots,k_n)$ be a partition we denote by $\sH(\bigcirc_\uk)$ the triply graded homology of the unknot labeled by $\uk$. We will compute this invariant when $\uk = (k)$ and $\uk = (1^2)$. Its Poincar\'e polynomial is denoted $\sP_{\bigcirc_\uk}(q,a,t)$ where the shifts $\{1\},[[1]]$ and $[1]$ are kept track of by $q,-a^2,t$ respectively.

\subsection{Cohomology of $\bigcirc_{(k)}$} \label{sec:example1}
If $k=1$ we have 
$$\sH(\bigcirc_{(1)}) \cong \pi_* \Delta^* (\O_\Delta) \cong \pi_* (\O_{\A_1} \oplus \O_{\A^1} [[1]] \{-2\}) \cong \k[x] \oplus \k[x] [[1]] \{-2\}$$  
where $\Delta$ and $\pi$ are the natural maps $\bullet \xleftarrow{\pi} \A_1 \xrightarrow{\Delta} \A_1 \times \A_1$. Hence
$$\sP_{\bigcirc_{(1)}}(q,a,t) = (1+q^{-2}+q^{-4}+\dots)(1-a^2q^{-2}) = \frac{1-a^2q^{-2}}{1-q^{-2}}.$$
Note that $\k[x] \cong \oplus_{i \ge 0} \k \{-2i\}$ which explains why it contributes $(1+q^{-2}+q^{-4}+\dots)$. More generally, $\A_k = \Spec \k[e_1,\dots,e_k]$ and a similar argument shows that 
$$\sH(\bigcirc_{(k)}) \cong \bigotimes_{\ell=1}^k \left(\k[e_\ell] \oplus \k[e_\ell][[1]]\{-2\ell\} \right).$$ 
It follows that 
\begin{align}\label{genser1}
\sP_{\bigcirc_{(k)}}(q,a,t) = \prod_{\ell=1}^k \frac{1-a^2 q^{-2 \ell}}{1-q^{-2\ell}}.
\end{align}

\subsection{Cohomology of $\bigcirc_{(1^2)}$} \label{sec:example2}
First, we need to explicitly identify the projector $\P^-$  which lives in $\Kom_*^-(D(\A_{1,1} \times \A_{1,1}))$. The braid element in this case is isomorphic to 
$$\T = [\E \F \la -1 \ra \rightarrow \id] \cong [\O_S \rightarrow \O_\Delta]$$
where $S$ is the variety described in the proof of Lemma \ref{lem:technical}. If $(x,y)$ are the coordinates of $\A_{1,1}$ then 
$$\T = \left[ \k[x,y] \otimes_{\k[e_1,e_2]} \k[x,y] \rightarrow \k[x,y] \right]$$
as $\k[x,y]$-bimodules (where $e_1=x+y$ and $e_2=xy$ are the usual elementary symmetric functions). Now, squaring and simplifying gives 
\begin{align*}
\T^2 &\cong [\E \F \E \F \la -2 \ra \rightarrow \E \F \la -1 \ra \oplus \E \F \la -1 \ra \rightarrow \id ] \\
&\cong [\E \F \la -3 \ra \xrightarrow{\sUupdot \sUdown - \sUup \sUdowndot} \E \F \la -1 \ra \xrightarrow{\Ucapr} \id] \\
&\cong [\O_S \{2\} \xrightarrow{x \otimes 1 - 1 \otimes x} \O_S \rightarrow \O_\Delta].
\end{align*}
The maps in the first and second lines above are encoded using the diagrammatics of \cite{KL}. The isomorphism between the first and second lines was proved in \cite[Section 10.2]{C1}. The isomorphism between the second and third lines follows from the fact that $\sUupdot \sUdown$ corresponds to $x \otimes 1$ and $\sUup \sUdowndot$ to $1 \otimes x$ (this follows form the action of the nilHecke defined in \cite{KL} or indirectly from the main result in \cite{C2}). 
Now, if we multiply again by $\T$ we get
\begin{align*}
\T^3 &\cong [\E \F \E \F \la -4 \ra \rightarrow \E \F \la -3 \ra \oplus \E \F \E \F \la -2 \ra \rightarrow \E \F \la -1 \ra \oplus \E \F \la -1 \ra \rightarrow \id] \\
&\cong [\E \F \la -5 \ra \xrightarrow{\sUupdot \sUdown + \sUup \sUdowndot - \sUup \smccbub{2} \sUdown} \E \F \la -3 \ra \xrightarrow{\sUupdot \sUdown - \sUup \sUdowndot} \E \F \la -1 \ra \xrightarrow{\Ucapr} \id] \\
&\cong [\O_S \{-4\} \xrightarrow{x \otimes 1 - 1 \otimes y} \O_S \{-2\} \xrightarrow{x \otimes 1 - 1 \otimes x} \O_S \rightarrow \O_\Delta].
\end{align*}
Here the $2$ beside the dot in the second line indicates that we add $2$ dots. The second line follows again from \cite[Section 10.2]{C1} while the third isomorphism is because $\sUup \smccbub{2} \sUdown$ is given by $1 \otimes 1 \mapsto (x+y) \otimes 1 = 1 \otimes (x+y)$. Continuing this way one finds that 
\begin{align}\label{eq:P}
\P^- = \lim_{\ell \rightarrow \infty} \T^{\ell} = \left[\dots \xrightarrow{g} \O_S \{-6\} \xrightarrow{f} \O_S \{-4\} \xrightarrow{g} \O_S \{-2\} \xrightarrow{f} \O_S \rightarrow \O_\Delta \right]
\end{align}
where the maps alternate between $f = x \otimes 1 - 1\otimes x$ and $g = x \otimes 1 - 1 \otimes y$. 

We need to compute $\sH(\bigcirc_{(1^2)}) = \Psi'_1 \Psi'_1 (\P^-)$. Now, using (\ref{eq:C}) and arguing as in the proof of Lemma \ref{lem:technical}, we find that $\Psi'_1(\P^-)$ is isomorphic to the complex 
\tiny
\begin{align*}
\xymatrix{
\dots \ar[r]^-{0} & \O_\Delta \otimes_\k \k[y] \{-4\} \ar@{}[d]^{\bigoplus} \ar[rr]^-{1 \mapsto x-y} & & \O_\Delta \otimes_\k \k[y] \{-2\} \ar@{}[d]^{\bigoplus} \ar[r]^-{0} & \O_\Delta \otimes_\k \k[y] \ar@{}[d]^{\bigoplus} \ar[r]^{\sim} & \O_\Delta \otimes_\k \k[y] \ar@{}[d]^{\bigoplus} \\
\dots \ar[r]^-{0} & \O_\Delta \otimes_\k \k[y] \{-8\} \ar[rr]^-{1 \mapsto x-y} & & \O_\Delta \otimes_\k \k[y] \{-6\} \ar[r]^-{0} & \O_\Delta \otimes_\k \k[y] \{-4\} \ar[r]^{1 \mapsto y} & \O_\Delta \otimes_\k \k[y]\{-2\} }
\end{align*}
\normalsize
where, going to the left, the differentials alternate. Now, consider the exact triangle 
$$\O_\Delta \otimes_\k \k[y] \{-2\} \xrightarrow{1 \mapsto x-y} \O_\Delta \otimes_\k \k[y] \rightarrow \O_\Delta.$$
Applying $\Psi'_1$ leaves us with 
$$\Psi'_1(\O_\Delta \otimes_\k \k[y] \{-2\}) \rightarrow \Psi'_1(\O_\Delta \otimes_\k \k[y]) \rightarrow \k[x][[1]]\{-2\} \oplus \k[x].$$
Thus, applying $\Psi'_1(\cdot)$ to $\Psi'_1(\P^-)$ gives us a complex isomorphic to 
$$
\begin{matrix}
\dots & 0 & \k[x] \{-6\} & 0 & \k[x]\{-2\} & 0 & 0 \\
\dots & 0 & \k[x]\{-8,-10\} & 0 & \k[x] \{-4,-6\} & 0 & \k[x]\{-2\} \\
\dots & 0 & \k[x] \{-12\} & 0 & \k[x]\{-8\} & 0 & \k[x] \{-4\}
\end{matrix}$$
where the top right entry in cohomology bidegree $(0,0)$. The generating series is then 
\begin{align*}\label{genser2}
\sP_{\bigcirc_{(1^2)}}(q,a,t) 
&= \frac{1}{1-q^{-2}} \frac{1}{1-q^{-4}t^2} (q^{-2}t^2 - q^{-2}a^{2} - q^{-4}a^2t^2 + q^{-4}a^4) \\ 
&= \frac{q^{-2}t^2(1-q^{-2}a^2)(1-a^2t^{-2})}{(1-q^{-2})(1-q^{-4}t^2)}.
\end{align*}

\section{Some remarks and speculation}\label{sec:speculation}

\subsection{$SL(N)$-homologies} 

In order to make the differential $d_N$ homogeneous one needs to kill the $[[\cdot]]$ grading. More precisely, one needs to set $[[-1]] = \{-2(-N+1)\}$. Since $[[1]]$ is recorded by $-a^2$ and $\{1\}$ by $q$ this means that the Euler characteristic $\chi_N(L)$ of $\sH_N(L)$ satisfies $\chi_N(L) = \sP_L(q,iq^{-N+1},-1)$. But
$$\sP_{\bigcirc_{(k)}}(q,iq^{-N+1},-1) = \prod_{\ell=1}^k \frac{1-q^{-2N+2-2\ell}}{1-q^{-2\ell}}$$
which, up to sign and a factor of $q$, equals $\qbins{N+k-1}{k}$. In particular, this means that if $L$ is a link labeled by $(k)$ then $\sH_N(L)$ categorifies the RT invariant of $SL(N)$ labeled by the representation $\Sym^k(\C^N)$. Moreover, the homology of the unknot in this case can be shown to be finite-dimensional homology. This implies (using conjugation invariance of the homology) that $\sH_N(L)$ is finite-dimensional for any $L$ labeled by partitions with only one part. 

\subsection{Batalin-Vilkovisky structures}

In Section \ref{sec:differentials} we defined the differential $d_N$ for $N>0$. This was based on the fact that for any algebra $\sA$ and $\sA$-bimodule $\sM$, $HH^*(\sA)$ acts on $HH_*(\sM)$. More generally, under fairly general hypothesis described in \cite[Section 1]{KK}, $HH^*(\sA)$ is a Gerstenhaber algebra and $HH_*(\sM)$ is a Batalin-Vilkovisky (BV) module.

Without going into details (see \cite{KK} for more) this equips $HH^*(\sA)$ with the usual cup product as well as a graded Lie algebra structure
$$\{\cdot,\cdot\}: HH^{p+1}(\sA) \otimes_\k HH^{q+1}(\sA) \rightarrow HH^{p+q+1}(\sA)$$
while $HH_*(\sM)$ carries the standard module structure as well as a graded Lie algebra module structure
\begin{equation}\label{eq:9}
\sL: HH^{p+1}(\sA) \otimes_\k HH_n(\sM) \rightarrow HH_{n-p}(\sM).
\end{equation}
When $p=-1$ we get a map 
$$HH^0(\sA) \otimes_\k HH_n(\sM) \rightarrow HH_{n+1}(\sM).$$
If $\sA$ is commutative then $HH^0(\sA) \cong \sA$ and for $f \in \sA$ we denote by $df$ the map $HH_n(\sM) \rightarrow HH_{n+1}(\sM)$ induced in (\ref{eq:9}) by $f$ (the conditions of being a BV-module implies that $d(fg) = fdg + gdf$). If we take $\sA = \k[x_1, \dots, x_k]$ then we obtain a map 
$$\sum_i d(x_i^N): HH_n(\sM) \rightarrow HH_{n+1}(\sM)$$
for any $\k[x_1, \dots, x_k]$-bimodule $\sM$. One would like this map to give a differential $d_{-N}$ which commutes with $d$ and so that, as in Theorem \ref{thm:2}, the total differential $d+d_{-N}$ defines a doubly graded link invariant $\sH_{-N}(L)$. This would give us a spectral sequence which commences at $\sH(L)$ and converges to $\sH_N(L)$ for any $N \in \Z$. 

On the other hand, if we take $p=0$ then we get a map 
\begin{equation}\label{eq:10}
HH^1(\sA) \otimes_\k HH_n(\sM) \rightarrow HH_n(\sM).
\end{equation}
Since $HH^1(\k[x]) \subset HH^1(\k[x_1, \dots, x_k])$ can be identified with the so-called Witt algebra one would hope that the resulting action from (\ref{eq:10}) agrees with the action of the Witt algebra defined in \cite{KR2} (see the introduction and Theorem 5.6 therein).

Finally, it is worth noting that in \cite[Section 2.3]{BF} and \cite[Corollary 1.1.3]{BG} one obtains a Gestenhaber algebra structure on $\Tor^*_X(\O_Y,\O_Z)$ whenever $Y,Z$ are smooth coisotropic subvarieties inside a smooth Poisson variety $X$ as well as a BV-module structure on $\Ext^*_X(\O_Y,\O_Z)$. In our case each term $\sM$ in the complex $\sT'_\beta$ is a (direct sum of) the structure sheaf of a non-smooth Lagrangian subvarieties inside $\A_\uk \times \A_\uk$ where the latter is equipped with the standard symplectic structure. This suggests that $HH_*(\sM)$ might carry the structure of a Gerstenhaber algebra and $HH^*(\sM)$ that of a BV-module over it.

\appendix 
\section{$\Psi$-functors}

In this section we suppose all varieties are smooth. However, we work over an arbitrary base ring and do not assume properness at any point. The results also hold if we equip all our varieties with an action of $\k^\times$ and work equivariantly. 

Fix a variety $Z$. For any two varieties $Y_1,Y_2$ we define
\begin{itemize}
\item $\Psi_Z: D(Y_1 \times Y_2) \rightarrow D((Y_1 \times Z) \times (Y_2 \times Z))$ and 
\item $\Psi'_Z: D((Y_1 \times Z) \times (Y_2 \times Z)) \rightarrow  D(Y_1 \times Y_2)$
\end{itemize}
via $\Psi_Z := \Delta_* \pi^*$ and $\Psi' := \pi_* \Delta^!$ where $\pi$ and $\Delta$ are the natural projection and diagonal inclusion maps 
$$\pi: Y_1 \times Z \times Y_2 \rightarrow Y_1 \times Y_2 \text{ and } \Delta: Y_1 \times Z \times Y_2 \rightarrow (Y_1 \times Z) \times (Y_2 \times Z).$$
Recall that if $i: Y_1 \rightarrow Y_2$ is an inclusion of smooth varieties then $i^!(\cdot) = i^*(\cdot) \otimes \omega_i [-c]$ where $\omega_i = \omega_{Y_1} \otimes \omega^\vee_{Y_2}|_{Y_1}$ and $c$ is the codimension of the inclusion.

\begin{Proposition}\label{prop:psi}
Let $Y_1,Y_2,Y_3,Z_1,Z_2,Z_3$ be six varieties and suppose 
$$\sP \in D(Y_1 \times Y_2), \sQ \in D(Y_2 \times Y_3), \sP' \in D(Z_1 \times Z_2) \text{ and } \sQ' \in D(Z_2 \times Z_3).$$
Then $(\sQ \boxtimes \sQ') \star (\sP \boxtimes \sP') \cong (\sQ \star \sP) \boxtimes (\sQ' \star \sP')$. 
\end{Proposition}
\begin{proof}
This is a fairly straightforward exercise with kernels which we leave up to the reader. 
\end{proof}

\begin{Corollary}\label{cor:psi}
Let $Y_1,Y_2,Y_3$ be three varieties and suppose 
$$\sP \in D(Y_1 \times Y_2) \ \ \text{ and } \ \ \sQ \in D(Y_2 \times Y_3).$$
Then $\Psi_Z(\sQ \star \sP) \cong \Psi_Z(\sQ) \star \Psi_Z(\sP)$. 
\end{Corollary}
\begin{proof}
This follows from Proposition \ref{prop:psi} by taking $Z_1=Z_2=Z$ and $\sP' = \sQ' = \Delta_* \O_Z$ because in this case $\Psi_Z(\cdot) \cong (\cdot) \boxtimes \O_Z$. 
\end{proof}

\begin{Proposition}\label{prop:trace1}
Let $Y_1,Y_2,Y_3$ be three varieties and suppose 
$$\sP \in D(Y_1 \times Y_2) \ \ \ \text{ and } \ \ \ \sQ \in D((Y_2 \times Z) \times (Y_3 \times Z)).$$ 
Then $\Psi'_Z(\sQ \star \Psi_Z(\sP)) \cong \Psi'_Z(\sQ) \star \sP \in D(Y_1 \times Y_3)$. 
\end{Proposition}
\begin{proof}
For $i,j \in \{1,2,3\}$ denote by $p_{ij}: Y_1 \times Y_2 \times Y_3 \rightarrow Y_i \times Y_j$ and 
$$p'_{ij}: (Y_1 \times Z) \times (Y_2 \times Z) \times (Y_3 \times Z) \rightarrow (Y_i \times Z) \times (Y_j \times Z)$$
the natural projections. We also denote by 
$$\pi_{ij}: Y_i \times Z \times Y_j \rightarrow Y_i \times Y_j \ \ \text{ and } \ \ \Delta_{ij}: Y_i \times Z \times Y_j \rightarrow (Y_i \times Z) \times (Y_j \times Z)$$
the projection and diagonal inclusion. Then 
\begin{eqnarray*}
\Psi'_Z(\sQ \star \Psi_Z(\sP)) 
&\cong& \pi_{13*} \Delta_{13}^* (\sQ \star (\Delta_{12*} \pi_{12}^* \sP)) \\
&\cong& \pi_{13*} \Delta_{13}^* (p'_{13*}(p'^*_{12} \Delta_{12*} \pi_{12}^* \sP \otimes p'^*_{23} \sQ)) \\
&\cong& \pi_{13*} \Delta_{13}^* (p'_{13*}(\Delta'_{12*} p''^*_{12} \pi_{12}^* \sP \otimes p'^*_{23} \sQ)) \\
&\cong& \pi_{13*} \Delta_{13}^* (p'_{13*} \Delta'_{12*} (p''^*_{12} \pi_{12}^* \sP \otimes \Delta'^*_{12} p'^*_{23} \sQ)) \\
&\cong& \pi_{13*} \Delta_{13}^* q_{13*} (p''^*_{12} \pi_{12}^* \sP \otimes p'''^*_{23} \Delta^*_{23} \sQ) \\
&\cong& \pi_{13*} q'_{13*} \Delta_Z^* (p''^*_{12} \pi_{12}^* \sP \otimes p'''^*_{23} \Delta^*_{23} \sQ) 
\end{eqnarray*}
where the third isomorphism follows from the commutative square in (\ref{eq:1}), the fourth via the projection formula, the fifth using $p'_{23} \Delta'_{12} = \Delta_{23} p'''_{23}$ where $p'''_{23}$ is the map 
$$(Y_1 \times Z \times Y_2) \times (Y_3 \times Z) \rightarrow Y_2 \times Z \times Y_3, \ \ \ \ (x_1,z,x_2,x_3,z') \mapsto (x_2,z,x_3)$$ 
and the last from the commutative square in (\ref{eq:2}). 

\begin{equation}\label{eq:1}
\xymatrix{
(Y_1 \times Z \times Y_2) \times (Y_3 \times Z) \ar[d]^{p''_{12}} \ar[rr]^{\Delta'_{12}} & & (Y_1 \times Z) \times (Y_2 \times Z) \times (Y_3 \times Z) \ar[d]^{p'_{12}} \\
Y_1 \times Z \times Y_2 \ar[rr]^{\Delta_{12}} & & (Y_1 \times Z) \times (Y_2 \times Z)
}
\end{equation}
\begin{equation}\label{eq:2}
\xymatrix{
(Y_1 \times Z \times Y_2) \times Y_3 \ar[rr]^{\Delta_Z} \ar[d]^{q'_{13}} & & (Y_1 \times Z \times Y_2) \times (Y_3 \times Z) \ar[d]^{q_{13}} \\
Y_1 \times Z \times Y_3 \ar[rr]^{\Delta_{13}} & & (Y_1 \times Z) \times (Y_3 \times Z)
}
\end{equation}
Now, $\pi_{13} q'_{13} = p_{13} (\pi_{12} \times \id_{Y_3})$ and $\pi_{12} p''_{12} \Delta_Z = p_{12} (\pi_{12} \times \id_{Y_3})$ so we get 
\begin{eqnarray*}
& & \pi_{13*} q'_{13*} \Delta_Z^* (p''^*_{12} \pi_{12}^* \sP \otimes p'''^*_{23} \Delta^*_{23} \sQ) \\ 
&\cong& p_{13*} (\pi_{12} \times \id_{Y_3})_* (\Delta_Z^* p''^*_{12} \pi_{12}^* \sP \otimes \Delta_Z^* p'''^*_{23} \Delta^*_{23} \sQ \otimes \omega_Z^\vee [- \dim Z ]) \\
&\cong& p_{13*} (\pi_{12} \times \id_{Y_3})_* ((\pi_{12} \times \id_{Y_3})^* p_{12}^* \sP \otimes \Delta_Z^* p'''^*_{23} \Delta^*_{23} \sQ \otimes \omega_Z^\vee [- \dim Z]) \\
&\cong& p_{13*} (p_{12}^* \sP \otimes (\pi_{12} \times \id_{Y_3})_* ((p'''_{23} \circ \Delta_Z)^* (\Delta^*_{23} \sQ \otimes \omega_Z^\vee [- \dim Z]))) \\
&\cong& p_{13*} (p_{12}^* \sP \otimes p_{23}^* \pi_{23*} \Delta^*_{23} \sQ) \\ 
&\cong& \Psi'_Z(\sQ) * \sP
\end{eqnarray*}
where the third isomorphism is via the projection formula and the fourth uses (\ref{eq:3}). 
\begin{equation}\label{eq:3}
\xymatrix{
(Y_1 \times Z \times Y_2) \times Y_3 \ar[d]^{p'''_{23} \circ \Delta_Z} \ar[rr]^{\pi_{12} \times \id_{Y_3}} & & Y_1 \times Y_2 \times Y_3 \ar[d]^{p_{23}} \\
Y_2 \times Z \times Y_3 \ar[rr]^{\pi_{23}} & & Y_2 \times Y_3
}
\end{equation}
The result follows. 
\end{proof}


\begin{thebibliography}{E-G-S}

\bibitem[AH]{AH}
M. Abel and M. Hogancamp, Categorified Young symmetrizers and stable homology of torus links II; {\sf arXiv:1510.05330} 

\bibitem[BF]{BF}
K. Behrend and B. Fantechi, Gerstenhaber and Batalin-Vilkovisky structures on Lagrangian intersections, Algebra, arithmetic, and geometry: in honor of Yu. I. Manin. Vol. I, \textit{Progr. Math.} \textbf{269} (2009), 1--47, Birkh\"auser, Boston, Inc., Boston, MA. 

\bibitem[BG]{BG}
V. Baranovsky and V. Ginzburg, Gerstenhaber-Batalin-Vilkovisky structures on coisotropic intersections, \textit{Math. Res. Lett} \textbf{17} (2010), no. 2, 221--229; {\sf arXiv:0907.0037} 

\bibitem[C1]{C1}
S. Cautis, Clasp technology to knot homology via the affine Grassmannian, \textit{Math. Ann.} \textbf{363} (2015), no. 3, 1053--1115; \textsf{arXiv:1207.2074}

\bibitem[C2]{C2}
S. Cautis, Rigidity in higher representation theory; {\sf arXiv:1409.0827}

\bibitem[CDK]{CDK}
S. Cautis, C. Dodd and J. Kamnitzer, Associated graded of Hodge modules and categorical $\sl_2$ actions; {\sf arXiv:1603.07402}

\bibitem[CK1]{CK1} 
S. Cautis and J. Kamnitzer, Knot homology via derived categories of coherent sheaves I, $\sl_2$ case, \textit{Duke Math. J.} \textbf{142} (2008), no. 3, 511--588. \textsf{math.AG/0701194}

\bibitem[CK2]{CK2} 
S. Cautis and J. Kamnitzer, Braiding via geometric categorical Lie algebra actions, \textit{Compos. Math.} \textbf{148} (2012), no. 2, 464--506; \textsf{arXiv:1001.0619}

\bibitem[CKM]{CKM}
S. Cautis, J. Kamnitzer, and S. Morrison, Webs and quantum skew Howe duality, \textit{Math. Ann.} \textbf{360} no. 1 (2014), 351--390; \textsf{arXiv:1210.6437}

\bibitem[CL]{CL}
S. Cautis and A. Lauda, Implicit structure in 2-representations of quantum groups, \textit{Selecta Math.} \textbf{21} no. 1 (2015), 201--244; {\sf arXiv:1111.1431}

\bibitem[DGR]{DGR}
N. Dunfield, S. Gukov and J. Rasmussen, The superpolynomial for knot homologies, \textit{Experiment. Math.} \textbf{15} (2006), no. 2, 129--159; {\sf arXiv:math/0505662}


\bibitem[GS]{GS}
S. Gukov and M. Sto\v{s}i\'{c}, Homological algebra of knots and BPS states, Proceedings of the Freedman Fest, \textit{Geom. Topol. Monogr.} \textbf{18}, 309--367, Geom. Topol. Publ., Coventry; {\sf arXiv:1112.0030}

\bibitem[H]{H}
M. Hogancamp, Stable homology of torus links via categorified Young symmetrizers I: one-row partitions; {\sf arXiv:1505.08148} 

\bibitem[K]{K}
M. Khovanov, Triply-graded link homology and Hochschild homology of Soergel bimodules, \textit{Intern. J. of Math.} \textbf{18} (2007), 869--885. {\sf arXiv:math/0510265} 

\bibitem[KL]{KL}
M. Khovanov and A. Lauda, A diagrammatic approach to categorification of quantum groups III, \textit{Quantum Topol.} \textbf{1}, Issue 1 (2010), 1--92; \textsf{arXiv:0807.3250}

\bibitem[KR1]{KR1}
M. Khovanov and L. Rozansky, Matrix factorizations and link homology II, \textit{Geom. Topol.} \textbf{12} (2008) 1387--1425; \textsf{math.QA/0505056}

\bibitem[KR2]{KR2}
M. Khovanov and L. Rozansky, Positive half of the Witt algebra acts on triply graded link homology, \textit{Quantum Topol.}, no. 4 (2016), 737--795; \textsf{arXiv:1305.1642}

\bibitem[KK]{KK}
N. Kowalzig and U. Krahmer, Batalin-Vilkovisky structures on Ext and Tor, \textit{J. Reine Angew. Math.}, \textbf{697} (2014), 159--219; \textsf{arXiv:1203.4984}

\bibitem[L]{L}
G. Lusztig, Introduction to quantum groups, \textit{Progress in Mathematics} \textbf{110}, Birkh\"auser Boston Inc., Boston, MA, 1993.

\bibitem[MSV]{MSV}
M. Mackaay, M. Sto\v{s}i\'{c} and P. Vaz, The 1,2-coloured HOMFLY-PT link homology, \textit{Trans. Amer. Math. Soc.} \textbf{363} (2011), 2091--2124; {\sf arXiv:0809.0193} 

\bibitem[Ra]{Ra}
J. Rasmussen, Some differentials on Khovanov-Rozansky homology, \textit{Geom. Topol.}, \textbf{19} (2015), no. 6, 3031--3104; {\sf arXiv:math/0607544}

\bibitem[Rou]{Rou}
R. Rouquier, 2-Kac-Moody Algebras; \textsf{arXiv:0812.5023}

\bibitem[Roz]{Roz}
L. Rozansky, An infinite torus braid yields a categorified Jones-Wenzl projector; \textit{Fund. Math.} \textbf{225} (2014), 305--326; \textsf{arXiv:1005.3266}

\bibitem[WW]{WW}
B. Webster and G. Williamson, A geometric construction of colored HOMFLYPT homology, \textit{Geom. Topol.} (to appear); {\sf arXiv:0905.0486}

\bibitem[W]{W}
P. Wedrich, Exponential growth of colored HOMFLY-PT homology; \textsf{arXiv:1602.02769}


\end{thebibliography}
\end{document}